\documentclass[a4paper,notitlepage, 8pt,reqno]{article}

  \usepackage{mathtext}
 \usepackage[cp1251]{inputenc}
  \usepackage[T2A]{fontenc}
 \usepackage[russian,english]{babel}
 \usepackage[all,arc,poly,2cell,curve,arrow,tips]{xy}

  \usepackage{enumerate, eucal, amsthm,amsmath, amssymb}

  \allowdisplaybreaks[4]

 \theoremstyle{definition}
 \newtheorem{definition}{Definition}

 \theoremstyle{plain}
 \newtheorem{theorem}{Theorem}
 \newtheorem*{thm*}{Tеорема}
 \newtheorem{proposition}{Proposition}
  \newtheorem*{prop*}{Предложение}
 
  \newtheorem*{cor*}{Следствие}
 \newtheorem{lemma}{Lemma}
  \newtheorem*{lem*}{Лемма}

 \theoremstyle{remark}
 \newtheorem{remark}[definition]{Remark}

  \newcounter{ab}
\author{D.~V.~Artamonov\footnote{ E-mail:~artamonov.dmitri@gmail.com}}

\title{The Schlesinger system and isomonodromic deformations of bundles with connections on Riemann surfaces}

\begin{document}

\maketitle

\begin{abstract}
We introduce a way of presentation of pairs $(E,\nabla)$, where $E$~
is a bundle on a Riemann surface and $\nabla$~is a logarithmic
connection in~$E$, which is based on a presentation of the surface
as a factor of the exterior of the unit disc.  In this presentation
we write the local equation of isomonodormic deformation of pairs
$(E,\nabla)$.  These conditions are written as a modified
Schlesinger system on a Riemann sphere (and in the typical case just
as an ordinary Schlesinger system) plus some linear system.
\end{abstract}

\section{Introduction}
\label{sec1}

Let   us be given a Fuchsian system of differential equations on a
Riemann sphere:
\begin{equation*}
\frac{dy}{dz}=\sum_i\frac{B_i}{z-a_i}y.
\end{equation*}
Let us change locations of singularities~$a_i$ in such a way that
the monodromy is preserved and the singularities do not confluence.
Then the residues~$B_i$ become  multivalued functions  of~$a_i$, and
the condition that the family of systems
\begin{equation*}
\frac{dy}{dz}=\sum_i\frac{B_i(a_1,\dots,a_n)}{z-a_i}y
\end{equation*}
is isomonodromic can be written as a system of nonlinear
differential equations for the functions~$B_i$. In particular if one
restricts to the Schlesinger deformations  (mention that in typical
case all deformations are Schlesinger, see~\cite{1}), than this
system is the Schlesinger system.

Take instead of the Riemann sphere a Riemann surface of positive
genus. In~ this case it is natural to consider not the deformations
of linear systems  (in other terms, connections in a trivial
bundle),  but isomonodromic deformations of bundles with
connections.  Both a connections and a bundle are allowed to change.
The motivation for this point of view can be found in the
introduction to the paper~\cite{2}.  It is also natural to allow to
change a module of a complex structure.  Such deformations were
considered by different authors in~\cite{3}--\cite{9} and also in
many other papers.

As usual special cases are considered, for example surfaces of
genus~1. Krichever wrote equations that describe deformations in the
general case ~\cite{5}.  His approach is based on meromorphic
triviality of bundles on Riemann surfaces,  however equations that
are obtained in~\cite{5},  differ much from the Schlesinger system.

In the case of genus~$1$ another approach to the description of
isomonodromic deformations is known: the elliptic Schlesinger system
(see for example papers~\cite{2},~\cite{10}),  which is some system
of equation describing isomonodromic deformation  of bundles with
connections on a torus; this system is a generalization of the
Sclesinger system . In the paper~\cite{2} the author says that is
desirable to generalize this construction to the case of higher
genus and to write in this case an analog of the Schlesinger system.
We fins this this generalization in the present paper. In particular
we prove that in the case of Riemann surfaces isomonodromic
deformations can be described by the Schlesinger system plus some
system of linear equation.

The paper is organized as follows. In section~\ref{sec2} the space
of parameters of derormations is described: this the Teichmuller
space with marked points  that are locations of singularities.  We
take the Teichmuller space not the space of modules by the following
reason:  in order to speak about monodromy matrix corresponding to
bypasses along canonical cuts, one must fix somehow  canonical cuts,
but the  pairs ``a complex structure$+$a system of canonical cuts''
modulo some equivalence form the Teichmuller space.  For an explicit
description of the space of parameters of the deformation we present
the Riemann surface as factor of the exterior of a unit disk by the
action of a Fuchsian group.  Then we choose in a canonical way the
fundamental polygon with $4g$  vertices ($g$~is the genus of the
surface). The singularities correspond to some points of the
polygon.

In the section~\ref{sec3}  we suggest a way of description of
bundles with connections on a Riemann surface.  They are described
by the following data: a form~$\omega$ on a Riemann sphere and a
collection of nondegenerate matrices $S_{x_0^1,x_0^i}$,
$i=2,\dots,4g$, where indices~$x_0^i$ correspond to the vertices of
the fundamental polygon. The form~$\omega$ is constructed as
follows.  There exists a factorizing map from the fundamental to the
Riemann surface. Take an inverse   of the bundle with connection on
the Riemann surface. The obtained bundle with connection on the
fundamental polygon is continued to a bundle with a connection on
the whole Riemann sphere with an additional singularity in the zero.
We fix then a meromorphic trivialization of this bundle on the
Riemann sphere, which is holomorphic on $\mathbb{C}\setminus\{0\}$.
In this trivialization the connection is defined by a form, this is
the form ~$\omega$ (see.~sec.~\ref{sec3.1}).

Explain the  construction of the matrices $S_{x_0^1,x_0^i}$.  The
bundle with a connection on the fundamental polygon is an inverse
image of  the bundle with the connection on the Riemann surface. All
vertices of the polygon are glued togather. Hence there exists an
operator that identifies stalks over vertices~$x_0^1$ and~$x_0^i$.
The matrix~$S_{x_0^1,x_0^i}$ is the matrix of this operator
(see~sec.~\ref{sec3.2}).

The constructed data is sufficient  for a reconstruct a bundle with
a connection on the surface (see~theorem~\ref{th1}). Note that the
procedure of construction  of the form and matrices is
non-canonical:  different forms and matrices can give equivalent
bundles.  But it is well-known that  all possible bundles do not
form any  ``good'' space (see~discussion in the paper~\cite{11}),
that is why there is no way of description of bundles with
connections.  Essentially we consider not the deformations of
bundles with connections, but the deformations of data, introduced
above (note that in the paper~\cite{5}
 instead of bundles with connections actually the parameters of meromorphic trivialization named Turin parameters are considered instead of).

In the section~\ref{sec4}  the Schlesinger isomonodromic
deformations of bundles with connections  on a Riemann surface are
defined and the equations are obtained, that describe evolution of
data, introduced in section~\ref{sec3}, under the Schlesinger
isomonodromic deformations. It is proved that the isomonodromic
deformations are described by a system of nonlinear equations for
the coefficients of the form ~$\omega$ (in typical case this just
the Schlesinger system)  and some linear system for the matrices
$S_{x_0^1x_0^i}$ (see.~Theorem~\ref{th2}).  The Sclesinger system
can be presented as a Hamiltonian system. Thus the approach
suggested in the present paper does not lead to the appearance of
new integrable systems.

The relations between the present approach and Krichever's approach
from~\cite{5} are considered in the section~\ref{sec5}.

\section{The space of parameters of defomations and the deformed objects}
\label{sec2}

\subsection{The space of parameters}
\label{sec2.1} Let $M$~be a Riemann surface of genus~$g>1$, the case
of genus~$g=0$, $g=1$ we consider trivial. The aim of the present
paper is to find an analog of the Schlesinger system (ordinary or
elliptic), well-known for genuses $g=0$, $g=1$, in the case of
higher genus.

Fix an initial point~$x_0$ on~$M$.

\begin{definition}
\label{def1} Let~$T$ be the Teichmuller space with~$n$ marked points
$a_1,\dots,a_n$, where $a_i\neq a_j$ for $i\neq j$. The
space~$\widetilde{T}$  of parameters of deformations is the
universal covering of the space~$T$.
\end{definition}

Take an image of a point $\tau\in\widetilde{T}$ in the space~$T$.
Then one can speak about marked points, corresponding to~$\tau$,
and also about a complex structure and a system of canonical cuts
corresponding to the point ~$\tau$, in other words, about a point of
the Teichmulle space  (without marked points), corresponding to the
point~$\tau$.  Below we shall call the Teichmuller space the
ordinary Teichmuller space without marked points, its points we
shall call marked Riemann surfaces.

The space~$\widetilde{T}_1$,  on which isomonodromic families of
pairs ``a bundle$+$a connection'', is constructed as follows.

\begin{definition}
\label{def2} Let~$T_1$ be the Teichmuller space with $n+1$ marked
points $z$, $a_1,\dots,a_n$, where $a_i\neq a_j$ for $i\neq j$.  The
space~$\widetilde{T}_1$ is the univeral covering of the space~$T_1$
by the variables~$a_i$.
\end{definition}

There exists a mapping $\widetilde{T}_1\to\widetilde{T}$,
``forgetting'' the marked point ~$z$.

\begin{definition}
\label{def3} Let $\tau\in\widetilde{T}$. Denote by
$\widetilde{T}_1|_{\tau}$ the preimage of the point~$\tau$ under the
mapping $\widetilde{T}_1\to\widetilde{T}$.
\end{definition}

The space $\widetilde{T}_1|_{\tau}$ can be viewed as a Riemann
surface with a complex structure, canonical cuts and marked points
$a_1,\dots,a_n$, defined by the point $\tau\in\widetilde{T}$.

The universal covering is taken to provide the global existence of
Schlesinger deformations for every initial condition (see~the
proposition~\ref{pr5} in  section~\ref{sec4}). In the present paper
we write only local equations of isomonodromic deformations. If one
considers only small changes of parameters one can use~$T$ instead
of~$\widetilde{T}$ as the space of parameters,  then parameters are
just locations of singularities and a point in the Teichmuller
space.  The deformed objects are pairs ``a bundle$+$a connection'',
not a form (a system of linear equations or, equivalently, a
connection in a trivial bundle as in the case of genus~0). When we
change positions of singularities both a bundle and a connection are
changing.  The evolution of a bundle and a connection is defined by
the change of a point in the Teichmuller space and of locations of
singularities.

\subsection{The fundamental polygon}
\label{sec2.2} Let $D$~be the exterior of the unit disc.  A point in
the Teichmuller can be defined by a collection of automorphisms
$Q_1,\dots,Q_{2g}\in\operatorname{Aut}D$, that satisfy the following
properties ~\cite{12}:

1) the equality $\prod_jQ_{2j}Q_{2j-1}Q^{-1}_{2j}Q^{-1}_{2j-1}=1$
holds;

2) the subgroup $G\subset\operatorname{Aut}D$, generated by
$Q_1,\dots,Q_{2g}$, is Fuchsian;

3) there exists a fundamental domain of the action of the subgroup
$G\subset\operatorname{Aut}D$ on~$D$, which does not intersect
$\partial D$.

The surface~$M$  is reconstructed as a factor of the exterior of the
unit disc under the action of the subgroup
$G\subset\operatorname{Aut} D$, generated by automorphisms
$Q_1,\dots,Q_{2g}$.  This action has a fundamental  domain, which is
non-euclidean polygon~$U$ with~$4g$ edges
 ("noneucledean"  means that every edge is a noneucledean line).  Canonical cuts are reconstructed as images of edges of the fundamental polygon.

Two collections $Q_1,\dots,Q_{2g}$ and $K_1,\dots,K_{2g}$ define the
same point in the Teichmuller space if and only if there exists an
automorphism~$Q$ such that
$K_1=QQ_1Q^{-1},\dots,K_{2g}=QQ_{2g}Q^{-1}$.  There exist a
normalized way to choose a collection of automorphisms
$Q_1,\dots,Q_{2g}$, such that it is reconstructed canonically from a
point in the Teichmuller space.  A traditional way of normalization
is described in~\cite{12}.  When a normalization is fixed one can
construct a canonical fundamental polygon. Its vertices~$x_0^i$
depend smoothly (but non complex-analytic) on a point in the
Teichmuller space.

Change a traditional way of normalization in such a way that
$x_0^1\equiv\infty$ become independent from the point in the
Teichmuller space. In order to do it let us change canonically an
automorphism $Q_z$~- of the set~$D$, that maps~$z$ to $\infty\in D$.
Define $Q_1,\dots,Q_{2g}$~as a normalized collection of generators
in a traditional sense, and let
 $x_0^1$~be a vertex of the fundamental polygon~$U$.
Take new generators
$Q_{x_0^1}Q_1Q^{-1}_{x_0^1},\dots,Q_{x_0^1}Q_{2g}Q^{-1}_{x_0^1}$.
They define a point in the Teichmuller space.  Also $Q_{x_0^1}(U)$
is a fundamental polygon for the action of the Fuchsian group given
by  a new set of generators. The first vertex of the polygon
$Q_{x_0^1}(U)$ is $\infty$.

We have proved.

\begin{proposition}
\label{pr1} There exists a  canonical way  of choosing a fundamental
polygon such that it (i.e. coordinates of its vertices) depend
smoothly (not complex-analytic) on a point in the Teichmuller space,
and one of its vertices is always $\infty$.
\end{proposition}

This polygon is denoted below as~$U$, and its vertices as~$x_0^i$,
$i=1,\dots,4g$, one has $x_0^1=\infty$. The marked points
$(z,a_1,\dots,a_n)$  become points in~$U$.

When we are studying  deformations in the section~\ref{sec4} we do
not allow the singularities to cross the canonical cuts, this is not
essential since we are considering only local deformations.

\section{Description of bundles with connections on a Riemann surface}
\label{sec3}

Let $E$~be a bundle on a surface~$M$ and $\nabla$~a connection
in~$E$ with singularities in $a_1,\dots,a_n\in M$. Suppose that an
initial point $x_0\in M$ is nonsingular. In the present section for
a pair $(E,\nabla)$ on a Riemann surface we construct a form on a
Riemann sphere and some matrices.  From them one can reconstruct a
pair $(E,\nabla)$ on a surface.  The construction of the form is
noncanonical since at some moment we fix a trivialization of some
bundle.

\subsection{The construction of the form on a Riemann sphere}
\label{sec3.1} Construct a form on a fundamental polygon.  We have
presented a marked Riemann surface as a factor of the fundamental
polygon~$U$. Let $(E_U,\nabla_U)$~be an inverse image on~$U$ of the
pair $(E,\nabla)$ under the factorization. Continue the pair
$(E_U,\nabla_U)$ until the pair
$(E_{\overline{\mathbb{C}}},\nabla_{\overline{\mathbb{C}}})$ on the
whole Riemann sphere. To do it let us calculate a monodromy of the
connection~$\nabla_U$ corresponding to the bypass along
$\gamma=\partial U$.  One can easily see that the monodromy equals
$M_{\gamma}=M_{a_1}\dots M_{a_n}$. We denote the monodromy of
~$\nabla$ corresponding to the bypass along~$\gamma$
as~$M_{\gamma}$, and the monodromy corresponding to the bypass
around $a_i$~as~$M_{a_i}$, $i=1,\dots,n$.

Take in the domain $\overline{\mathbb{C}}\setminus U$ a trivial
bundle~$E'$. Take as ~$\nabla'$  a connection with the only
singularity and the monodromy  ~$M_{\gamma}$ of the bypass around
zero. Thus on the boundary~$\partial U$ the bundles~$E_U$ and~$E'$
are trivial and connections~$\nabla_U$ and~$\nabla'$ in them have
the same monodromy.  From here we conclude that we can glue pairs
$(E_U,\nabla_U)$ and $(E',\nabla')$ into a pair
$(E_{\overline{\mathbb{C}}},\nabla_{\overline{\mathbb{C}}})$ on the
whole Riemann sphere.  It can be obtained by gluing them over
horizontal sections over~$\partial U$. Let us describe this
procedure of gluing since we shall use it several  times.

\begin{proposition}
\label{pr2} Let $V$~be a domain on  a Riemann surface and  $\gamma$~
be a nonclosed curve without self intersections that cuts the domain
into two parts~$V'$ and ~$V''$.  Let us be given two pairs:
$(E',\nabla')$ on~$V'$ and $(E'',\nabla'')$ on $V''$ without
singularities on~$\gamma$. Fix an identifications $E'_P=E''_P$ over
some point $P\in\gamma$. Then there exists a uniquely defined
procedure of gluing of pairs $(E',\nabla')$ and $(E'',\nabla'')$
into a pair $(E,\nabla)$ on $V$. If a curve is closed, then the
gluing is possible if and only if the monodromies of ~$\nabla'$
and~$\nabla''$ along~$\gamma$ are the same.
\end{proposition}

\begin{proof}[\textup{of the proposition is well-know,  we omit it}]
\end{proof}

Every bundle on a Riemann sphere is meromorphically trivial.
Moreover there exists a meromorphic trivialization, which is
holomorphic on $\overline{\mathbb{C}}\setminus\{0\}$. Fix such a
trivialization. The sections become holomorphic vector-columns. The
connection~$\nabla_{\mathbb{C}}$ can be defined using a form
\begin{equation}
\omega=\biggl(\frac{C_k}{z^k}+\dots+\frac{C_1}{z}+
\sum_i\frac{B_i}{z-a_i}\biggr)\,dz \label{eq1}
\end{equation}
with a regular singularity in zero.

\begin{definition}
\label{def4} Let $(E_U,\nabla_U)$ be the inverse image of the pair
$(E,\nabla)$ under the factorization $U\to M$. Then~$\omega$ is the
form of the connection~$\nabla_U$ in such a trivialization of~$E_U$.
It is of the type~\eqref{eq1} with a regular singularity in zero.
\end{definition}

Note that such a trivialization is not unique.

\begin{remark}
For a typical  monodromy matrices and typical positions of
singularities one can take such a trivialization that the form is
written as
\begin{equation}
\omega=\biggl(\frac{C_1}{z}+
\sum_{i=1}^{n}\frac{B_i}{z-a_i}\biggr)\,dz. \label{eq2}
\end{equation}
In this situation $C_1=-\sum_{i=1}^nB_i$. If we put  $a_0=0$
and~$B_0=-\sum_{i=1}^n B_i$, then
\begin{equation*}
\omega=\sum_{i=0}^{n}\frac{B_i}{z-a_i}\,dz.
\end{equation*}
\end{remark}

\subsection{The construction of the matrices of gluing operators $S_{x_0^1x_0^i}$}
\label{sec3.2} Introduce additional objects. Using them and the
form~$\omega$ one can reconstruct a pair $(E,\nabla)$ on a Riemann
surface. To reconstruct a bundle on a Riemann surface we need
operators $S_{z,z'}\colon E_{U,z}\mapsto E_{U,z'}$ from a stalk of
the bundle~$E_U$ over a point~$z$ to the stalk of~$E_U$ over a
point~$z'$. These operators  are defined in the following way for
every ordered pair of points $z,z'\in\partial U$, that are glued
under the factorization $U\to M$.

\begin{definition}
\label{def6} The bundle~$E_U$ is an inverse image under the
factorization of the bundle~$E$ on~$M$, hence there exist
isomorphisms of stalks $E_{U z}\to E_{Z}$ and $E_{Uz'}\to E_{Z}$.
Define~$S_{z,z'}$ as $E_{U,z}\to E_{Z}\to E_{U,z'}$, where the
second mapping is the inverse to ~$E_{U,z'}\to E_{Z}$.
\end{definition}

But it is excess to know all operators $S_{z,z'}$. Let
$x^1_0,\dots,x^{4g}_0$~be the vertices of the fundamental polygon.
Below we show that it is sufficient to know only the
operators~$S_{x^1_0,x^i_0}$.  Since the trivialization of the
bundle~$E_U$ is fixed, we speak below about the
matrices~$S_{x^1_0,x^i_0}$.

\begin{definition}
\label{def7} The matrices $S_{x^i_0,x^j_0}$ are matrices of
operators, that glue the stalks over points  $x_0^i$,~$x_0^j$ in the
sense of the definition~\ref{def6}.
\end{definition}

Thus,  for the bundle with a connection on a Riemann surface~$M$ and
an initial point~$x_0$ we have constructed~$\omega$ of
type~\eqref{eq1} with singularities $a_i\in U$, $i=1,\dots,n$, and a
regular singularity at zero, and matrices $S_{x^1_0,x^i_0}$,
$i=1,\dots,4g$.

\subsection{The reconstruction of a bundle with a connection from the form~$\omega$ and matrices~$S_{x^1_0,x^i_0}$} \label{sec3.3}
At first we suppose that we are given a form and matrices that are
obtained from a pair on a Riemann surface and consider the procedure
of the reconstruction  of a pair on a Riemann surface. Then we
investigate the question when such a reconstruction is possible.

At first we construct a pair $(E_U,\nabla_U)$ on the fundamental
polygon: a bundle~$E_U$ is a trivial bundle $U\times\mathbb{C}^p$,
and $\nabla_U$~is a connection in it, defined by the form ~$\omega$.
Reconstruct operators~$S_{z,z'}$ for every pair of points  $z,z'\in
\partial U$ that are glued under the factorization in the surface.
The points $z,z'\in\partial U$ belong to edges $x^i_0 x^{i+1}_0$
and $x^{j+1}_0 x^j_0$ that are glued (the order means that the edges
with the opposite factorization).

\begin{proposition}
\label{pr3} Let $Y_1$~be a matrix, whose columns are horizontal
sections of the bundle~$E_U$ over the edge $x^i_0 x^{i+1}_0$ with
the initial condition $Y_1(x_0^i)=E$. Let $Y_2 $~-be a matrix, whose
columns are horizontal sections of~$E_U$ over the edge $x^{j+1}_0
x^j_0$ with the initial condition
$Y_2(x_0^{j+1})=S_{x_0^ix_0^{j+1}}=S_{x_0^1x_0^{i}}^{-1}S_{x_0^1x_0^{j+1}}$.
Then $S_{z,z'}=Y_2(z')Y_1(z)^{-1}$.
\end{proposition}

\begin{proof}
Since the edges $x^i_0 x^{i+1}_0$ and $x^{j+1}_0 x^j_0$ are glued
into one cut, the stalks of~$E_U$ over points of these edges must be
glued into stalks of~$E$. The matrices~$Y_1$ and~$Y_2$  are
transformed into two collections of horizontal sections over this
cut. The initial conditions for these horizontal sections coincide,
since the collection of sections $Y_1(x_0^i)$ of the bundle~$E_U$
over the point~$x_0^i$ is identified with the collection of sections
$S_{x_0^ix_0^{j+1}}Y_1(x_0^i)=Y_2(x_0^{j+1})$ over the
point~$x_0^{j+1}$. Then the collections of sections~$Y_1$ and~$Y_2$
must be glues together over the whole cut.  It follows that if the
points $z\in x^i_0x^{i+1}_0$ and $z'\in x^{j+1}_0x^j_0$ are glued
under the factorization, then $S_{z,z'}Y_1(z)=Y_2(z')$. The
proposition is proved.
\end{proof}

The total space of the bundle~$E$ is obtained from the total
space~$E_U$ in the following way: if the points $z,z'\in\partial U$
are glued under the factorization into the Riemann surface then, we
glue the stalks~$E_{U,z}$ and~$E_{U,z'}$ using the
operator~$S_{z,z'}$.

The connection $\nabla$ in~$E$ is reconstructed automatically.
Indeed, $\operatorname{int}U$ is mapped biholomorphicly onto some
open dense subspace ~$U'$ in~$M$. Hence the connection~$\nabla_U$
uniquely defines a connection~$\nabla$ in $E|_{U'}$. Since the
set~$U'$ is dense and in the set $\partial U'$ there are no
singularities, the connection is uniquely defined  on the whole
surface.

\goodbreak

We have proved the following statement.

\begin{proposition}
\label{pr4} From a form~$\omega$ of type~\eqref{eq1}  with
singularities in the fundamental polygon~$U$ and matrices
$S_{x^1_0,x^i_0}$, $i=2,\dots,4g$, obtained from a pair $(E,\nabla)$
on a  surface, the pair $(E,\nabla)$ is reconstructed as follows:

1) we construct a pair $(E_U,\nabla_U)$~-which is a trivial bundle
on~$U$ with a connection defined by the form $\omega${\rm;}

2) using the rule described in the proposition~\ref{pr3} using the
matrices $S_{x^1_0,x^i_0}$ we reconstruct the matrices~$S_{z,z'}$
for all pairs of points $z,z'\in\partial U$, that are glued under
the factorization;

3) the total space~$E$ is obtained from the total space of~$E_U$
using the following rule: if $z,z'\in \partial U$ are glued together
under the factorization into the Riemann surface, then we glue the
stalks~$E_{U,z}$ and~$E_{U,z'}$ using the operators~$S_{z,z'}$;

4) the connection $\nabla$ in~$E$ is reconstructed automatically.
\end{proposition}

Express the monodromy of~$\nabla$ along the cuts using the
form~$\omega$and matrices~$S_{x^1_0,x^i_0}$. Denote by
$(E_U,\nabla_U)$ a trivial bundle on~$U$ with a connection defined
by~$\omega$.

\begin{definition}
\label{def8} Take in a stalk~$E_{U,x_0^1}$ over the point~$x_0^1$ a
base $e^1_1,\dots,e^1_p$. Then in the stalk~$E_{U,x_0^i}$ over a
point~$x^0_i$ we obtain a base
$S_{x^1_0,x^i_0}e^1_1,\dots,S_{x^1_0,x^i_0}e^1_p$. Such  a system of
bases in the stalks~$E_{U,x_0^i}$, $i=1,\dots,4g$ we call a coherent
system of bases.
\end{definition}

When we identify $E_{U,x_0^i}=E_{x^0}$, all these bases are
identified with one base in  $E_{x^0}$, we denote it
$e_1,\dots,e_p$.

Take as a base in $E_{U,x_0^1}$ a standard base
\begin{equation*}
e_1^1=(1,0,\dots,0),\qquad\dots,\qquad e_p^1=(0,0,\dots,1).
\end{equation*}
Take a coherent system of bases in $E_{U,x_0^i}$, $i=1,\dots,4g$.
This gives as a base $e_1,\dots,e_p$ in ~$E_{x^0}$. Find the
monodromy matrices in this base.

\begin{lemma}
\label{lem1} Consider horizontal sections $y_1,\dots,y_p$ starting
at the point $x^1_0$ then going along the curve $x^1_0x^2_0\dots
x^{i-1}_0x^{i}_0$ with the initial condition $y_k(x_0^1)=e^1_k$,
$k=1,\dots,p$. Write then in a matrix $Y=(y_1,\dots,y_p)$. Then the
monodormy matrix corresponding to the bypass along the curve,
obtained from $x^1_0x^2_0\dots x^{i-1}_0x^{i}_0$ after the
factorization, equals $S^{-1}_{x^{1}_0,x^{i}_0}Y(x_0^{i})$.
\end{lemma}

\begin{proof}
Indeed, $S^{-1}_{x^{1}_0,x^{i}_0}Y(x_0^{i})$~is a matrix in a base
$y_k(x_0^1)=e_k^1$, $k=1,\dots,p$,  of the operator that firstly
does the horizontal transportation of sections along the curve from
the stalk over~$x_0^1$to the stalk over~$x_0^i$, and then identifies
the stalks as under the factorization into~$E$.  By definition this
matrix is the monodromy matrix in the base $e_1,\dots,e_p$ along the
curve that we obtain under the factorization $x^1_0x^2_0\dots
x^{i-1}_0x^{i}_0$. The lemma is proved.
\end{proof}

Below, when we say ``the monodromy of the bypass along the loop
$x^1_0x^2_0\dots x^{i-1}_0x^{i}_0$'',  we shall mean the monodromy
of the bypass along the loop that we obtain from $x^1_0x^2_0\dots
x^{i-1}_0x^{i}_0$ under the factorization.

If we know monodromies of the bypasses along all loops
$x^1_0x^2_0\dots x^{i-1}_0x^{i}_0$, then we know the monodromy of
the bypass along each loop $x_0^jx_0^{j+1}$.  Thus we shall write
the condition that the monodromy is preserved we shall write the
condition that the monodromy of the bypasses along all loops
$x^1_0x^2_0\dots x^{i-1}_0x^{i}_0$ is preserved.  Below we need also
an expression for the monodromy of the bypass along the loop
$x_0^{i}x_0^{i+1}$.

\begin{lemma}
\label{lem2} The monodromy matrix of $\nabla$ of the bypass along
the loop that we get from $x_0^{i}x_0^{i+1}$ after the factorization
is written as follows: take horizontal sections
$\tilde{y}_1,\dots,\tilde{y}_p$ of the pair $(E_U,\nabla_U)$
starting from~$x^i_0$ going along $x_0^ix_0^{i+1}$ such that
$\tilde{y}_k(x_0^i)=e_k^i$, $k=1,\dots,p$, write them in a matrix
$\widetilde{Y}=(\tilde{y}_1,\dots,\tilde{y}_p)$ (note that
$\widetilde{Y}(x_0^i)=S_{x_0^1x_0^i}$). Then the monodromy along
$x_0^{i}x_0^{i+1}$ equals
$\widetilde{Y}^{-1}(x_0^{i})S^{-1}_{x^{i}_0,x^{i+1}_0}
\widetilde{Y}(x_0^{i+1})=
S_{x_0^1x_0^{i+1}}^{-1}\widetilde{Y}(x_0^{i+1})$.
\end{lemma}

\begin{proof}[\textup{is analogous to the proof of the previous lemma}]
\end{proof}

Up to now we have suggested that the form~$\omega$ and matrices
$S_{x^1_0,x^i_0}$ are obtained from a pair $(E,\nabla)$ on a Riemann
surface. Let us now give an answer to the following natural
question. Let us be given a Riemann surface presented as a factor of
the exterior of the unit disc. Let us be given a form~\eqref{eq1} on
the Riemann sphere,  such that all its singularities, except may be
zero,  belong to the fundamental polygon~$U$, let us be given a
collection of nondegenerate matrices~$S_{x^1_0,x^i_0}$,
$i=1,\dots,4g$. In the Proposition~\ref{pr4}  we have described a
procedure how to reconstruct a bundle with a connection on the
Riemann surface from this data. The question is for which data this
procedure is correct?

\begin{lemma}
\label{lem3} A necessary and sufficient condition for the
possibility of reconstruction of $(E,\nabla)$ using the procedure
from the proposition~\ref{pr4} is the following: let
$x_0^ix_0^{i+1}$ and $x_0^{j+1} x_0^j$ be glued into one canonical
cut on the Riemann surface. Then
\begin{equation}
Y_1^{-1}(x_0^{i})S^{-1}_{x^{i}_0,x^{i+1}_0}Y_1(x_0^{i+1})=
Y_2^{-1}(x_0^{j+1})S^{-1}_{x_0^{j+1}\!,\,x^j_0}Y_2(x_0^j),
\label{eq3}
\end{equation}
where $Y_1$~is a matrix whose columns are horizontal sections of the
pair $(E_U,\nabla_U)$ along $x_0^ix_0^{i+1}$ with the initial
condition $Y_1(x_0^i)=S_{x_0^1x_0^i}$, and $Y_2$~is a matrix whose
columns are horizontal sections of the pair $(E_U,\nabla_U)$ along
$x_0^{j+1}x_0^j$ with the initial condition
$Y_2(x_0^{j+1})=S_{x_0^1x_0^{j+1}}$.
\end{lemma}

\begin{proof}
Let us prove that the condition is necessary. Since $x_0^ix_0^{i+1}$
and $x_0^{j+1}x_0^j$ are glued into one cut, the monodromy of the
bypass along this cut can be calculated using $x_0^ix_0^{i+1}$ or
$x_0^{j+1}x_0^j$, but the result must be the same. The monodromy
calculated using $x_0^ix_0^{i+1}$ equals
$Y_1^{-1}(x_0^{i})S^{-1}_{x^{i}_0,x^{i+1}_0}Y_1(x_0^{i+1})$, the
monodromy calculated using $x_0^{j+1}x_0^j$, equals
$Y_2^{-1}(x_0^{j+1})S^{-1}_{x^{j+1}_0,x^j_0}Y_2(x_0^j)$. The
condition~\eqref{eq3} just says that these two expressions are
equal.

Now let us prove that the condition is sufficient.  We can always
construct a pair $(E_U,\nabla_U)$~ which is a trivial bundle with a
connection defined by the form~$\omega$. Since $\operatorname{int}U$
is mapped biholomorphicly onto some open dense subset $U'\subset M$,
we obtain a bundle with a connection on~$U'$.  We need to glue it
into a bundle with a connection on the whole surface.

Take a point $P=x_0$ and consider its small neighborhood ~$O$.  At
the point~$P$ all cuts meet. In the point~$P$ the gluing procedure
is already defined~$S_{x_0^1,x_0^i}$. Using the
proposition~\ref{pr2}, we can glue along the horizontal section
along cuts that are contained in~$O$. As a result we obtain a bundle
with a connection $(E,\nabla)$ over this small neighborhood~$O$,  we
need to glue along the rest parts of the cuts that are not contained
in~$O$. Again we shall glue along the horizontal sections.  The rest
part of a cut is a curve, whose ends~$P_1$ and~$P_2$ belong
to~$\partial O$. Take one of its ends~$P_1$ as an initial point and
do the gluing along the horizontal sections.  We need to check that
this procedure is correct in the end point~$P_2$.

Now we are in the situation described in the Proposition~\ref{pr2}:
we take as~$\gamma$ the whole cut (with its  part that belongs
to~$O$) as $V$~ a small neighborhood of the cut. We have already a
bundle with a connection in~$V$ outside~$\gamma$  and a gluing in
$\gamma\cap O$. The correctedness of gluing along~$\gamma$ is
equivalent to the coincidence of monodromies of glued connections
along~$\gamma$. But the condition~\eqref{eq3} just expresses this
coincidence. The lemma is proved.
\end{proof}

In this section we have proved.

\begin{theorem}
\label{th1} For the bundle with a connection $(E,\nabla)$ on a
Riemann surface we have constructed the following data:

1) a form~\eqref{eq1} on sphere, all singularities of the form,
except may be zero, belong to the fundamental polygon~$U$, and zero
is a regular singularity;

2) a collection of matrices $S_{x^1_0,x^{i}_0}$, where$x_0^i$~are
vertices of the fundamental polygon, $i=1,\dots,4g$, and
$S_{x^i_0,x^{i+1}_0}=S_{x_0^1x_0^{i+1}}S_{x_0^1x_0^{i}}^{-1}$
satisfy the condition~\eqref{eq3} for all pairs of edges
$x_0^ix_0^{i+1}$ and $x_0^{j+1}x_0^j$, that are glued into one cut.
The matrices $Y_1$,~$Y_2$ are the same as in Lemma~\ref{lem3}.

The inverse is true: using such data one can construct a bundle with
a connection on a surface.
\end{theorem}

\section{Isomonodomic deformation}
\label{sec4}

In the previous sections pair a $(E,\nabla)$ on the surface we  have
constructed a form~\eqref{eq1}  on a Riemann sphere with a regular
singularity in zero and matrices $S_{x^1_0,x^i_0}$.   Now we show
how the isomonodromic deformations of pairs $(E,\nabla)$ are
described using such correspondence.

Let $(E^1,\nabla^1)$~-be a pair on $ \widetilde{T}_1$. For
$\tau\in\widetilde{T}$ denote by $(E^1,\nabla^1)|_{\tau}$
 a restriction of the bundle with a connection $(E^1,\nabla^1)$ to the subspace $\widetilde{T}_1|_{\tau}$ (see~Definition~\ref{def3}).

\begin{definition}
\label{def9} An isomonodormic family is pair $(E^1,\nabla^1)$
on~$\widetilde{T}_1$ such that:

1) a pair $(E^1,\nabla^1)$has singularities on hypersurfaces $z=a_i$
(more precise, on hypersurfaces in~$\widetilde{T}_1$, that are
preimages of hypersurfaces $z=a_i$ in~$T_1$);

2) for all $\tau\in \widetilde{T}$ pairs $(E^1,\nabla^1)|_{\tau}$
have the same monodormy.
\end{definition}

Let us be given a point $\tau_0\in \widetilde{T}$. Let
$(E,\nabla)$~be a pair on a marked Riemann surface, corresponding
to~$\tau_0$, the singularities of $\nabla$ correspond to marked
points of~$\tau_0$.

\begin{definition}
\label{def10} We say that an isomonodromic family $(E^1,\nabla^1)$
describes a deformation of a pair $(E,\nabla)$, if
$(E^1,\nabla^1)|_{\tau_0}=(E,\nabla)$.
\end{definition}

\begin{definition}
\label{def11} A family $(E^1,\nabla^1)$ is called Schlesinger, if
for a fixed point in the Teichmuller space in some neighborhood of
the hypersurface $z=a_i$ a connection $\nabla^1$ is written in local
coordinates as a form of type
\begin{equation*}
\frac{B_i}{\zeta-a_i}\,d(\zeta-a_i)+h(\zeta,a_i),
\end{equation*}
where $h(\zeta,a_i)$~is a holomorphic form, $B_i$~are holomorphic
functions of~$a_i$.
\end{definition}

Let us state a result about global existence of deformations.

\begin{proposition}
\label{pr5} For every logarithmic initial pair $(E,\nabla)$ at
${t_0\in\widetilde{T}}$  there exists a unique its continuation  to
the Schlesinger isomonodromic family $(E^1,\nabla^1)$.
\end{proposition}

\begin{proof}
Since this proposition is well-known let us give only a sketch of a
proof. We must construct a pair $(E^1,\nabla^1)$
on~$\widetilde{T}_1$.  First of all note that there is an
isomorphism $\pi_1(M\setminus
\{a^0_1,\dots,a^0_n\})\to\pi_1(\widetilde{T}_1)$ (here
$a_1^0,\dots,a_n^0$~are initial positions of singularities); this is
a corollary of the homotopic equivalence .  Using the R\"orl
construction~\cite{13}, one cam construct a pair $(E^1,\nabla^1)$
on~$\widetilde{T}_1$ outside small neighborhoods of hypersurfaces
$z=a_i$. The problem of  a construction of $(E^1,\nabla^1)$in a
neighborhood of a hypersurface $z=a_i$ is local, we can use the
analogous construction in the case of the Riemann sphere. Thus we
obtain a pair $(E^1,\nabla^1)$ on~$\widetilde{T}_1$.
\end{proof}

Establish a relation between a pair $(E^1,\nabla^1)$
on~$\widetilde{T}_1$ and a family of forms from the
theorem~\ref{th1}. Consider the Schlelinger family $(E^1,\nabla^1)$
and a point $t^0\in\widetilde{T}_1$. Denote ~as $\tau^0$
singularities, corresponding to~$t^0$ as $a_1^0,\dots,a_n^0$, and a
point in the Teichmullr space, corresponding to~$t^0$. Let
$W_{a_i^0}$~be a sufficiently small neighborhood of the
point~$a^0_i$ (such that $W_{a_i^0}\cap W_{a_j^0}=\varnothing$, if
$i\neq j$), $i=1,\dots,n$, and $V_{\tau^0}$~ be a sufficiently small
neighborhood of the point~$\tau^0$ in the Teichmuller space.

\begin{proposition}
\label{pr6} There exists a form $\omega^1$ on the space
\begin{equation*}
\bigl\{(z,a_1,\dots,a_n,\tau)\colon z\in\overline{\mathbb{C}},\;
a_i\in W_{a_i^0},\; i=1,\dots,n,\; t\in V_{\tau^0}\bigr\},
\end{equation*}
with the following properties:

1) the following presentation takes place
\begin{equation}
\omega^1=\frac{C_k}{z^k}\,dz+\dots+\frac{C_1}{z}\,dz+
\sum_{i=1}^n\frac{B_i}{z-a_i}\,d(z-a_i); \label{eq4}
\end{equation}

2) if we fix a point~$t$, which is sufficient close to ~$t^0$, and
consider a pair $(E^1,\nabla^1)|_{t}$ and a form $\omega$,
corresponding to the pair $(E^1,\nabla^1)|_{t}$ by
theorem~\ref{th1}, then the form~$\omega$  can be obtained by fixing
singularities in the form~$\omega^1$ in singularities, corresponding
to the point ~$t$\,\footnote{It is important to note that the
form~$\omega^1$ does not contain coordinates on~$V_{\tau^0}$ and
differentials of these coordinates.}.
\end{proposition}

\begin{proof}
Take an intersection of all fundamental polygons that correspond to
all points in the Teichmuller space, that belong to~$V_{\tau}$. Let
$O$~be an open neighborhood of this intersection, whose boundary is
smooth simple curve. Suppose that the neighborhood~$V_{\tau^0}$ is
so small that $W_{a_i^0}\subset O$ for all $i=1,\dots,n$.

There exist a mapping
\begin{equation}
f\colon O\times W_{a_1^0}\times\dots\times W_{a_n^0}\times
V_{\tau}\to T_1, \label{eq5}
\end{equation}
which is defined in the following way. Take a point
\begin{equation*}
(z,a_1,\dots,a_n,t)\in O\times W_{a_1^0}\times\dots\times
W_{a_n^0}\times V_{\tau}.
\end{equation*}
Using a point~$t$ in the Teichmuller space, one can reconstruct a
pair $G_t\subset\operatorname{Aut} D$ (remind that~$D$ is the
exterior of the unit disc). Denote as $z/G_t$ the image of $z\in
O\subset D$ on the marked Riemann surface under the factorization
$D\to M$ under the action of~$G_t$. Then~$f$ maps the point
$(z,a_1,\dots,a_n,t)$ to the point
$(t,z/G_t,a_1/G_t,\dots,a_n/G_t)\in T_1$. Since the image of this
mapping is small enough, the mapping~\eqref{eq5}  is well defined.
Note that this mapping is not holomorphic, but it becomes
holomorphic if we fix a point in the Teichmuller space.

Take in inverse image of $(E^1_O,\nabla^1_O)$ of the pair
$(E^1,\nabla^1)$ under the mapping~$f$.  If a point in the
Teichmuller space is fixed, this is a holomorphically trivial bundle
with a connection.

When the parameters $a_1,\dots,a_n,\tau$ are fixed one can continue
$(E^1_O,\nabla^1_O)$ from the subset~$O$  to the whole Riemann
sphere~$\overline{\mathbb{C}}$. Note that the
connection~$\nabla^1_O$ depends holomorphically from
$a_1,\dots,a_n$, and the bundle~$E^1_O$ does not depend on
$a_1,\dots,a_n$. Hence for fixed  $\tau\in V_{\tau^0}$ the
holomorphic pair $(E^1_O,\nabla^1_O)$ can be continued from $O\times
W_{a_1^0}\times\dots\times W_{a_n^0}\times \{\tau\}$ to a
holomorphic pair on $\overline{\mathbb{C}}\times
W_{a_1^0}\times\dots\times W_{a_n^0}\times\{\tau\}$.

By the Grotendick-Birkhoff theorem with parameters~\cite{1}  there
exists a meromorphic trivialization of this bundle, which is
holomorphic on $\overline{\mathbb{C}}\setminus\{0\}\times
W_{a_1^0}\times\dots\times W_{a_n^0}\times \{\tau\}$. The
form~$\omega^1$ is the form of the connection~$\nabla^1_O$in this
trivialization. Note that~$\omega^1$ can (in a non-holomorphic way)
depend on~$\tau$. Remind that $(E^1,\nabla^1)$~is a Schlesinger
family.  Hence for a fixed~$\tau$ the form~$\omega^1$ is of type
\begin{equation*}
\frac{C_k}{z^k}\,dz+\dots+\frac{C_1}{z}\,dz+
\sum_{i=1}^n\frac{B_i}{z-a_i}\,d(z-a_i)+\sum_{i=1}^n D_i\,da_i.
\end{equation*}

Choose another trivialization in which coefficients at $da_i$
vanish.  It is constructed in the following way. The family
$(E^1,\nabla^1)$ is isomonodromic, hence the family of
forms~$\omega^1$ is also isomonodromic. This is equivalent of the
fact that $d\omega^1=\omega^1{\wedge}\,\omega^1$, i.e. the
form~$\omega^1$ is integrable. Let $(p\times p)$  matrix
$Y_0(z,a_1,\dots,a_n)$ be a solution of $dY_0=\omega^1Y_0$ (the
differential is taken by the variables $a_1,\dots,a_n$)  with the
initial condition $Y_0(\infty,a^0_1,\dots,a^0_n)=I$. Let
$Y_0^{\infty}=Y_0(\infty,a_1,\dots,a_n)$. Define a new
trivialization of $E^1_O$ on
$\overline{\mathbb{C}}\setminus\{0\}\times
W_{a_1^0}\times\dots\times W_{a_n^0}\times \{\tau\}$, by acting on
the old base by the matrix $(Y_0^{\infty})^{-1}$, to obtain a new
base in every stalk. The form of the connection~$\nabla^1_O$ in this
new trivialization is the new form~$\omega^1$.

By construction in this new trivialization there exist a solution of
the system $dY=\omega^1Y$ such that $Y(\infty,a_1,\dots,a_n)\equiv
I$ for all $a_1,\dots,a_n$. Since
\begin{equation*}
\frac{\partial Y}{\partial a_i}\biggr|_{z=\infty}=
\frac{B_i}{z-a_i}\biggr|_{z=\infty}+D_i=D_i,
\end{equation*}
we have $D_i=0$. Thus in new trivialization of $E^1_O$ on the space
$\overline{\mathbb{C}}\setminus\{0\}\times
W_{a_1^0}\times\dots\times W_{a_n^0}\times \{\tau\}$ the
form~$\omega^1$ is written as~\eqref{eq4}.

Note that the change of the point in the Teichmuller space leads
only to the change of the fundamental polygon.  But  the
trivialization of~$E_O^1$ in section~\ref{sec3.1} does not depend on
its precise shape. Thus the trivialization of~$E^1_O$ can be chosen
such that the equality~\eqref{eq4} for~$\omega^1$ takes place on the
whole space $\overline{\mathbb{C}}\setminus\{0\}\times
W_{a_1^0}\times\dots\times W_{a_n^0}\times V_{\tau}$. The
Proposition is proved.
\end{proof}

Now let us write the equations of isomonodromic deformations.  The
first group of equations describes the Schlesinger deformations of
the form~\eqref{eq1}. Obviously when the pair $(E,\nabla)$ deforms
isomonodromically, the form $\omega$ also deforms isomonodromically.

Consider the Schlesinger deformations.  Then the deformations
of~$\omega$ are defined by the form~$\omega^1$ of type~\eqref{eq4}.
The form~$\omega^1$ defines an isomonodromic deformation if and only
if $d\omega^1=\omega^1\wedge\omega^1$. Note that~$\omega^1$ does not
contain coordinates on the Tecihmuller space and differentials of
these coordinates,  hence the differential is taken only by
variables~$z$,~$a_i$.

\begin{proposition}
\label{pr7} The equations of isomonodromic deformations of $\omega$
are
\begin{align}
&dB_i=-\sum_{\substack{j=1,\\ j\neq i}}^n
\frac{[B_i,B_j]}{a_i-a_j}\,d(a_i-a_j)+ \frac{\partial C_1}{\partial
a_i}\,da_i, \label{eq6}
\\
&\frac{\partial C_2}{\partial a_i}- \frac{\partial C_1}{\partial
a_i}a_i=-[B_i,C_1],\qquad\dots,\qquad \frac{\partial
C_{l+1}}{\partial a_i}- \frac{\partial C_l}{\partial
a_i}a_i=-[B_i,C_l],
\label{eq7}\\
&\dots\dots\dots\dots\dots\dots\dots\dots\dots\dots\dots\dots,
\nonumber
\\
&-\frac{\partial C_k}{\partial a_i}a_i=-[B_i,C_k],\qquad
i=1,\dots,n. \label{eq8}
\end{align}
In the typical case $C_l=0$,~$l>1$, if we put $a_0=0$,
$B_0=C_1=-\sum_{i=1}^nB_i$, the system above turns into an ordinary
Schlesinger system
\begin{equation}
B_i=-\sum_{\substack{j=0,\\ j\neq i}}^n
\frac{[B_i,B_j]}{a_i-a_j}\,d(a_i-a_j). \label{eq9}
\end{equation}
Here $i=0,1,\dots,n$.
\end{proposition}

\begin{proof}
Explicitly the equality $d\omega^t=\omega^t\wedge\, \omega^t$ can be
written as follows. At first let us write the left hand side:
\begin{align*}
&\sum_{i=1}^n \frac{dB_i}{z-a_i}\,d(z-a_i)+
\sum_{l=1}^k\frac{dC_l}{z^l}\,dz= \sum_{i,j=1}^n\frac{\partial
B_i}{\partial a_j} \frac{1}{z-a_i}\,dz\wedge d(z-a_i)-{}
\\
&-\sum_{i,j=1}^n\frac{\partial B_i}{\partial a_j}
\frac{1}{z-a_i}\,d(z-a_j)\wedge d(z-a_i)+
\sum_{l=1}^k\sum_{i=1}^n\frac{\partial C_l}{\partial a_i}
\frac{1}{z^l}\,dz\wedge d(z-a_i).
\end{align*}
Now let us write the right hand side:
\begin{align*}
&\biggl(\sum_{i=1}^n\frac{B_i}{z-a_i}\,d(z-a_i)+
\sum_{l=1}^k\frac{C_l}{z^l}\,dz\biggr)\wedge
\biggl(\sum_{j=1}^n\frac{B_j}{z-a_j}\,d(z-a_j)+
\sum_{l=1}^k\frac{C_l}{z^l}\,dz\biggr)=
\\
&=\sum_{\substack{i,j=1,\\ i\neq j}}^n\frac{B_iB_j}{(a_i-a_j)}
\biggl(\frac{1}{z-a_i}-\frac{1}{z-a_j}\biggr)\, d(z-a_i)\wedge
d(z-a_j)+{}
\\
&\phantom{={}}+
\sum_{l=1}^k\sum_{i=1}^n\frac{[B_i,C_l]}{(z-a_i)z^l}\,
d(z-a_i)\wedge dz.
\end{align*}
The forms $d(z-a_i)\wedge d(z-a_j)$ Ё $dz\wedge d(z-a_i)$ are
independent~\cite{1},
 so the coefficients at these forms on the right and on the left must coincide. The coincidence of the coefficients at the form
$dz\wedge d(z-a_i)$ gives the equation
\begin{equation*}
\sum_{j=1}^n \frac{\partial B_i}{\partial a_j}\frac{1}{z-a_i}+
\sum_{l=1}^k\frac{\partial C_l}{\partial a_i}\frac{1}{z^l}=
-\sum_{l=1}^k\frac{1}{z^l}\frac{[B_i,C_l]}{z-a_i}.
\end{equation*}
Multiply it by $(z-a_i)$, one gets
\begin{equation*}
\sum_{j=1}^n \frac{\partial B_i}{\partial a_j}+
\sum_{l=1}^k\biggl(\frac{1}{z^{l-1}}-a_i\frac{1}{z^l}\biggr)
\frac{\partial C_l}{\partial a_i}=
-\sum_{l=1}^k\frac{1}{z^l}[B_i,C_l].
\end{equation*}
Consider the coefficient at powers of~$1/z$: for~$z^0$ we obtain the
equation
\begin{equation}
\sum_{j=1}^n \frac{\partial B_i}{\partial a_j}+ \frac{\partial
C_1}{\partial a_i}=0, \label{eq10}
\end{equation}
for $1/z,\dots,1/z^l$ we obtain equations~\eqref{eq7} and at last
for $1/z^k$ we obtain the equation~\eqref{eq8}.

The coincidence of coefficients at forms $d(z-a_j)\wedge d(z-a_i)$
gives equations that does not contain~$C_k$, they are equivalent to
the following equations
\begin{equation*}
\frac{\partial B_i}{\partial a_j}= \frac{[B_i,B_j]}{a_i-a_j},\qquad
i\neq j.
\end{equation*}
These equations together with~\eqref{eq10} can be written together
as the equation~\eqref{eq6}.

In the case $C_2=\dots=C_k=0$  the derivation above just reproduce
the derivation of the Schlesinger equations for the deformations of
the form~\eqref{eq2} (see~\cite{1}). In particular in notations
$a_0=0$, $B_0=C_1=-\sum_{i=1}^nB_i$ we obtain an ordinary
Schlesinger system~\eqref{eq9}. The proposition is proved.
\end{proof}

The equations~\eqref{eq6}--\eqref{eq8}  are uniquely solvable for
every initial conditions just because they describe the Schlesinger
deformations of the connection~$\nabla_{\overline{\mathbb{C}}}$ on a
Riemann sphere.

Now let us write the second group of equation that describe the
evolution of matrices~$S_{x^1_0,x^i_0}$. In lemma~\ref{lem1} it was
proved that $S^{-1}_{x^{1}_0,x^{i}_0}Y(x_0^{i})$ is the monodromy
matrix  along the loop, that we obtain after the factorization
$x^{1}_0x^{2}_0\dots x^{i-1}_0x^{i}_0$ (here~$Y$ is the same as in
Lemma~\ref{lem1}).

Take an isomonodromic family of solutions $dY=\omega^1Y$  such that
in the initial position of singularities~$a_j$ in the point
$z=\infty$ the matrix~$Y$ is identical (remind that in this equation
the differential is taken only by variables~$z$ and~$a_j$,
$j=1,\dots,n$). The form $\omega^1$, defined in~\eqref{eq4}, is such
that the matrix~$Y$ in $z=\infty$ is constantly identical.  This
follows from the fact that $\frac{\partial Y}{\partial
a_i}\bigr|_{z=\infty}= \frac{B_i}{z-a_i}\bigr|_{z=\infty}=0$. That
is why we can assume that the matrix~$Y$, through which the
monodromy in Lemma ~\ref{lem1} is expressed, satisfies $dY=\omega^1
Y$.

The monodromy along each cut must be conserved. Hence
$S^{-1}_{x^{1}_0,x^{i}_0}Y(x_0^{i})=\operatorname{const}$ or
equivivalently, $S_{x^{1}_0,x^{i}_0}=
\operatorname{const}^{-1}\,\cdot\, Y(x_0^{i})$. Hence the
matrix~$S_{x^{1}_0,x^{i}_0}$ satisfies the same equation, as
$Y(x_0^{i})$, but with another initial condition.  In other words
$d_{x_0^i,a_j}S_{x_0^ix_0^{i+1}}= \omega^1|_{z\mapsto
x_0^i}S_{x_0^ix_0^{i+1}}$.  Here $\omega^1|_{z\mapsto x_0^i}$ is the
form~$\omega^1$, in which the variable~$z$ is replaced to~$x_0^i$.
Namely
\begin{equation*}
\omega^1|_{z\mapsto x_0^i}=
\sum_{j=1}^n\frac{B_j}{x_0^i-a_j}\,d(x_0^i-a_j)+
\biggl(\frac{C_{k}}{{x_0^i}^{k}}+\dots+
\frac{C_{0}}{{x_0^i}}\biggr)\,dx_0^i.
\end{equation*}

Now let us state the main result of the present paper.

\begin{theorem}
\label{th2} For every initial logarithmic pair $(E,\nabla)$ on a
marked Riemann surface there exists a unique Schlesinger
isomonodromic family  $(E^1,\nabla^1)$. In terms of data from the
Theorem~\ref{th1} The Schlesinger isomonodromic deformations are
locally described as follows.

1.  The vertices of the fundamental polygon change their posiitons
according to the change of a point in the Teichmuller space.

2. The evolution of the coefficients $C_l$, $B_i$ of the
form~\eqref{eq1} is given by the equations from the
proposition~\ref{pr7}. In the case $C_l=0$, $l>0$, in notations
$a_0=0$, $B_0=-\sum_{i=1}^nB_i$ these equations are the Schlesinger
system for~$B_i$, $i=0,\dots,n$.

3. The evolution of matrices $S_{x^{1}_0,x^{i}_0}$ is described by
equations $dS_{x^{1}_0,x^{i}_0}=\omega^1S_{x^{1}_0,x^{i}_0}$,
where~$\omega_1$ is the form~\eqref{eq4}, in which instead of~$z$ we
write the variable~$x_0^i$. The differential in the left side is
taken by the variables~$a_j$,~$x_0^i$.
\end{theorem}

One ca easily see that the bundle in the pair $(E,\nabla)$ changes,
thus the case of higher genus differs much from the case of
genus~$0$,  where it is natural to consider connetions in a fixed
trivial bundle.

\section{Relation to the Krichiver's approach}
\label{sec5}

Let us give a short comparision of the approach to the description
of isomonodromic deformations of bundles with connection suggested
in the present paper with the approach suggested by
Kricheve~\cite{5}.  More precise we are going to show how using the
parameters that Krichever has used one can reconstruct the form and
matrices from Theorem~\ref{th1}.

Restrict ourself to the case when the marked Riemann surface is
fixed. The reason is that in~\cite{5}  the complex-analytic
coordinates on the Teichmuller space are used,  and in the present
paper we take real-analytic coordinates  on the Teichmuller spaces
$x_0^i$~ which are positions of vertices of the fundamental polygon.
The matrices$S_{x_0^1,x_0^i}$ are complex analytic functions of the
variables~$x_0^i$, $i=2,\dots,4g$. Thus in order to establish  a
relation between the Krichiver's approach and our approach in
general case one must be able to express explicitly positions of
vertices of the fundamental polygon through the complex-analytic
coordinates on the Teichmuller space.  This problem for~$g>1$ is
extremely difficult (see for example a close to this problem
paper~\cite{14}).

Thus let  a  marked Riemann surface be fixed.  Let us be given a
stable bundle with a logarithmic connection $(E,\nabla)$, where~$E$
has rank~$p$and degree $pg$. In the paper~\cite{5}  for the bundle
with a connections some parameters are constructed, these parameters
can be divided into two groups.  The parameters from the first group
describe the bundle (the Turin parameters),  and parameters of the
second group describe a connection in it (see~\S\,2 in~\cite{5}).

The Turin parameters are defined as follows.  There exists a
meromorphic trivialization of a stable bundle of degree $g$. This is
a collection of holomorphic sections $\psi_1,\dots,\psi_p$. In
stalks of~$E$ over all points except $\gamma_1,\dots,\gamma_m$
(where $m=pg$), they form a base in the stalk of~$E$. In the stalks
over $\gamma_1,\dots,\gamma_m$ these sections are linearly
dependent. In the typical case the rank of the span of sections
$\psi_1,\dots,\psi_p$ in these points equals~$p-1$, that is there
exists a unique linear relation
$a_1^i\psi_1(\gamma_i)+\dots+a_p^i\psi_p(\gamma_i)=0$. The Turin
parameters are: the collection of points $\gamma_1,\dots,\gamma_m$
and the vectors $(a_1^i,\dots,a_p^i)$, $i=1,\dots,m$. In the
meromorphic trivialization $\psi_1,\dots,\psi_p$ to the
connection~$\nabla$ there corresponds the form~$\widetilde{\omega}$
with apparent singularities in points $\gamma_1,\dots,\gamma_m$and a
trivial monodormy of the bypass around them. In the paper~\cite{5}
it is suggested that the typical case takes place and these apparent
singularities of~$\widetilde{\omega}$ are simple poles.

The  parameters of the second  group describe the
connection~$\nabla$. Among these parameters there are those that
describe the behavior of~$\widetilde{\omega}$ in a neighborhood
of~$\gamma_i$ (see.~\S\,2 and the Lemma~2.2 in~\cite{5}).  In
particular using these parameters and the Turin parameters one can
reconstruct a residue of the form~$\widetilde{\omega}$ in
~$\gamma_i$. Also there are parameters that define behavior
of~$\widetilde{\omega}$  in a neighborhood of singularities of the
connections (see~\S\,4 in~\cite{5}).  Among these parameters there
are the position of singularities and singular parts of
~$\widetilde{\omega}$ in these neighborhoods.

In the present paper for a pair $(E,\nabla)$ we have constructed a
form~$\omega$ and matrices~$S_{x_0^1,x_0^i}$. Note that in contrast
with Krichever's parameters one can not say that the form~$\omega$
defines a connection and matrices $S_{x_0^1,x_0^i}$~define a bundle,
since in the procedure of reconstruction of the bundle the
from~$\omega$ participates.

Now let us establish a relation between descriptions of
$(E,\nabla)$, suggested in the present paper and the description
suggested by Krichever.  At first we construct a form and matrices
to the meromorphically trivialized bundle.  Using the
trivialization~$E|_U$, defined by sections $\psi_1,\dots,\psi_p$, we
obtain that the from is an inverse image of~$\widetilde{\omega}$ on
the fundamental polygon (we denote it also as~$\widetilde{\omega}$)
and, since the bundle is trivial, identity matrices
$\widetilde{S_{x_0^1,x_0^i}}=I$.

Note that the residues of the form $\widetilde{\omega}$ in all
singularities are    contained in Krichevers parameters, thus the
form ~$\widetilde{\omega}$ on the fundamental polygon can be
explicitly reconstructed from Krichevers parameters.  Now describe
how we can reconstruct the form~$\omega$ and
matrices~$S_{x_0^1,x_0^i}$ from~$\widetilde{\omega}$
and~$\widetilde{S_{x_0^1,x_0^i}}=I$. To reconstruct~$\omega$, we
must apply to~$\widetilde{\omega}$
 a gauge transformation~$\Gamma$, which must remove apparent singularities of~$\widetilde{\omega}$ in points of the fundamental polygon,  corresponding to the points
$\gamma_1,\dots,\gamma_m$ of the surface. The
transformation~$\Gamma$ can have singularities only in these points
and $\Gamma(x_0^1)=I$. Such transformation~$\Gamma$ can be found
using the residues of~$\widetilde{\omega}$.

To reconstruct $S_{x_0^1,x_0^i}$, we need to act on
$\widetilde{S_{x_0^1,x_0^i}}=I$ by the same gauge transformation by
the rule $\widetilde{S_{x_0^1,x_0^i}}= I\mapsto
S_{x_0^1,x_0^i}=\Gamma(x_0^i)$. Indeed,
$\widetilde{S_{x_0^1,x_0^i}}$ is a matrix of the operator that
identifies the stalks, hence after the base change according to the
definition ~\ref{def6} we obtain the matrix
$S_{x_0^1,x_0^i}=\Gamma(x_0^i)\widetilde{S_{x_0^1,x_0^i}}\Gamma(x_0^1)^{-1}$.
But two matrices in the right side are identity matrices, hence we
obtain the requied expression.

\subsection*{Acknolegment}
I would like to thank V.~A.~Poberejnij and R.~R.~Gontzov. The work
was supported by the  program of support of young researches
(grant~MK-4270.2011.1).

\end{document}